\documentclass[fontsize=11pt,a4paper,DIV12]{scrartcl}

\usepackage[utf8]{inputenc} 

\usepackage{rsfso}

\usepackage{color}
\usepackage{cite}
\usepackage{enumerate}
\usepackage{xcolor}
\usepackage{graphicx}
\usepackage{amsthm}
\usepackage{amssymb}
\usepackage{amsmath}
\usepackage{bbm}
\usepackage{complexity} 
\usepackage{wasysym} 

\newcommand{\crown}[2]{\mathit{Crown}(#1,#2)}

\graphicspath{{figs/}}
\input{pdffig.sty}

\usepackage[font=small,labelfont=bf]{caption}
\usepackage[font=small,labelfont=normalfont,labelformat=simple]{subcaption}

\usepackage{hyperref}

\newtheorem{theorem}{Theorem}
\newtheorem{proposition}[theorem]{Proposition}

\newtheorem{question}{Question}

\newtheorem{conjecture}[theorem]{Conjecture}

\newtheorem{lemma}[theorem]{Lemma}

\usepackage{xfrac}
\usepackage{enumerate}

\def\AA{{\cal A}}

\def\ZZ{\mathbb{Z}}
\def\GC{\ensuremath{G^\circ}}
\def\indeg{\textrm{indeg}}

\def\inst#1{$^{#1}$}

\date{}

\title{Coloring Circle Arrangements:
	New~4-Chromatic~Planar Graphs\footnote{
		M.-K.\ Chiu was supported  by ERC StG 757609.
		S.\ Felsner was supported by DFG Grant  FE~340/13-1.
		M.\ Scheucher was supported 
		by the DFG Grant SCHE~2214/1-1.
		R.\ Steiner was supported by DFG-GRK 2434 and by an ETH Zurich Postdoctoral Fellowship.
		B.\ Vogtenhuber was supported by the FWF project \mbox{I 3340-N35}.
		This work was initiated at a workshop of the collaborative DACH project
		\emph{Arrangements and  Drawings} in Malchow, Mecklenburg-Vorpommern.
		We thank the organizers and all the participants for the  inspiring atmosphere. }
}

\begin{document}
	
\author{
	Man-Kwun Chiu\inst{1}
	\and
	Stefan Felsner\inst{2}
	\and
	Manfred Scheucher\inst{2}
	\and
	Felix Schröder\inst{2}
	\and
	Raphael Steiner\inst{2,4}
	\and
	Birgit Vogtenhuber\inst{3}
}

\maketitle
\vspace{-1.cm}
\begin{center}
	{\footnotesize
		\inst{1}
		Institut f\"ur Informatik,\\
		Freie Universit\"at Berlin, Germany,\\
		\texttt{chiumk@zedat.fu-berlin.de}
		\\\ \\
		\inst{2}
		Institut f\"ur Mathematik,\\
		Technische Universit\"at Berlin, Germany,\\
		\texttt{\{felsner,scheucher,fschroed,steiner\}@math.tu-berlin.de}
		\\\ \\
		\inst{3}
		Institute of Software Technology,\\
		Graz University of Technology, Austria,\\
		\texttt{bvogt@ist.tugraz.at}
		\\\ \\
		\inst{4} Institute of Theoretical Computer Science,\\
		Department of Computer Science,\\
		ETH Z\"{u}rich, Switzerland,\\
		\texttt{raphaelmario.steiner@inf.ethz.ch}
		\\\ \\
	}
\end{center}

\begin{abstract}

Felsner, Hurtado, Noy and Streinu (2000) conjectured that 
arrangement graphs of simple great-circle arrangements have chromatic number at	most $3$. 
Motivated by this conjecture, we study the colorability of arrangement graphs for different classes of arrangements of (pseudo-)circles. 

In this paper the conjecture is verified for \emph{$\triangle$-saturated} pseudocircle 
arrangements, i.e., for arrangements where one color class of the 2-coloring
of faces consists of triangles only, as well as for further classes of (pseudo-)circle arrangements. These results are complemented by a construction which
maps $\triangle$-saturated arrangements with a pentagonal face
to arrangements with 4-chromatic 4-regular arrangement graphs.  
This \emph{corona} construction has similarities with the
\emph{crowning} construction introduced by Koester (1985).
Based on exhaustive experiments with small arrangements we propose
three strengthenings of the original conjecture.

We also investigate fractional colorings. 
It is shown that the arrangement graph of every arrangement $\mathcal{A}$ of pairwise intersecting pseudocircles is ``close'' to being $3$-colorable. 
More precisely, the fractional chromatic number $\chi_f(\mathcal{A})$ of the arrangement graph is bounded from above by $\chi_f(\mathcal{A}) \le 3+O(\frac{1}{n})$, where $n$ is the number of pseudocircles of $\mathcal{A}$. 
Furthermore, we construct an infinite family of $4$-edge-critical $4$-regular planar graphs which are fractionally $3$-colorable. 
This disproves a conjecture of Gimbel, K\"{u}ndgen, Li, and Thomassen (2019).
\end{abstract}

\section{Introduction}
\label{sec:intro}

An \emph{arrangement of pseudocircles} is a family of simple closed curves on the
sphere or in the plane such that each pair of curves intersects at most twice.
Similarly, an \emph{arrangement of pseudolines} is a family of $x$-monotone curves
such that every pair of curves intersects exactly once.  An arrangement is
\emph{simple} if no three pseudolines/pseudocircles intersect in a common
point and \emph{intersecting} if every pair of pseudolines/pseudocircles
intersects.  Given an arrangement of pseudolines/pseudocircles, the
\emph{arrangement graph} is the planar graph obtained by placing vertices at
the intersection points of the arrangement and thereby subdividing the
pseudolines/pseudocircles into edges. 

A \emph{proper coloring} of a graph assigns a color to each vertex such that
no two adjacent vertices have the same color.  The \emph{chromatic number}
$\chi$ of a graph is the smallest number of colors needed for a proper coloring of the graph. 
For an arrangement $\mathcal{A}$, we denote the chromatic number of (the arrangement graph of) $\mathcal{A}$ by $\chi(\mathcal{A})$.

The famous Four Color theorem and also Brook's theorem imply the $4$-colorability of
planar graphs with maximum degree $4$; hence also every arrangement graph is properly 4-colorable. 
This motivates the following question: which arrangement graphs have chromatic number 4 and which can be properly colored with fewer than four colors? 

There exist arbitrarily large non-simple line arrangements that require 4
colors.  For example, the construction depicted in Figure~\ref{fig:nonsimple} 
contains the Moser spindle 
as subgraph which has chromatic number 4. Hence the construction cannot be properly 3-colored.  Using an inverse central (gnomonic) projection which maps lines to great-circles,
one gets a non-simple arrangement $\AA$ of great-circles with $\chi(\AA)=4$ for any such line arrangement. 
Therefore, we restrict our attention to simple arrangements in the following.

Koester~\cite{Koester1985} presented a simple arrangement $\mathcal{A}$ of 7 circles with
$\chi(\mathcal{A})=4$ in which all but one pair of circles intersect; see
Figure~\ref{fig:Koester2} in Section~\ref{subsec:corona}.  
Moreover, there also exist simple intersecting arrangements that require 4 colors.  
We invite the reader to verify this property for the example depicted in Figure~\ref{fig:n5_pwi_chi4_1}.

\begin{figure}[htb]
  \hbox{} \hfill
	\begin{subfigure}[b]{.4\textwidth}
		\centering
		\includegraphics[page=3]{figs/nonsimple}
		\caption{}
		\label{fig:nonsimple}
	\end{subfigure}
	\hfill
	\begin{subfigure}[b]{.4\textwidth}
		\centering
		\includegraphics[width=0.9\textwidth]{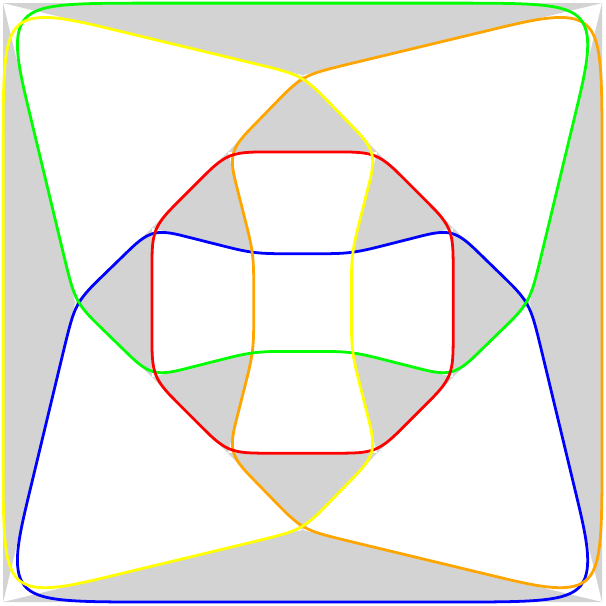}
		\caption{}
		\label{fig:n5_pwi_chi4_1}
	\end{subfigure}
	\hfill
	\hbox{}
	\caption{ \subref{fig:nonsimple}~A 4-chromatic non-simple line
          arrangement.  The red subarrangement not intersecting the Moser
          spindle (highlighted blue) can be chosen arbitrarily.
          \subref{fig:n5_pwi_chi4_1}~A simple intersecting arrangement of 5
          pseudocircles with $\chi=4$ and $\chi_f=3$.  }
	\label{fig:nonsimchi4}
\end{figure}

In 2000, Felsner, Hurtado, Noy and Streinu~\cite{FHNS2000} (cf.\
\cite{FelsnerHNS2006}) studied arrangement graphs of
pseudoline and pseudocircle arrangements. They obtained results regarding
connectivity, Hamiltonicity, and colorability of those graphs. In
that work, they also stated the following conjecture:

\begin{conjecture}[Felsner et al.\ \cite{FHNS2000,FelsnerHNS2006}]
  \label{conj:FHNS}
  The arrangement graph of every simple arrangement of great-circles on the sphere is
  3-colorable.
\end{conjecture}
      
While this conjecture is fairly well known (cf.\
\cite{OPG_website_3colorconj,Kalai_blog_coloring_problems,Wagon2002} and
\cite[Chapter~17.7]{Wagon2010book}) there has been little progress in the last
20 years.  Aichholzer, Aurenhammer, and Krasser verified the conjecture for up
to 11 great-circles \cite[Chapter~4.6.4]{Krasser2003}.  
They did not explicitly mention ``non-realizable'' arrangements, i.e., arrangements of pseudocircles that cannot be realized by great-circles despite fulfilling all necessary combinatorial properties of great-circle arrangements (see below for details).
We have re-generated the data from \cite[Chapter~4.6.4]{Krasser2003} for arrangements of up to 11
great-circles (cf.\ \cite{ScheucherSS2020}) and verified the conjecture also for
non-realizable arrangements of the same size, by this confirming it for all
arrangements of up to 11 \emph{great-pseudocircles}.
\emph{Arrangements of great-pseudocircles} are defined as arrangements of pairwise intersecting
pseudocircles where along each pseudocircle, the sequence
of the $2n-2$ intersections with the other pseudocircles is ($n-1$)-periodic.
Equivalently, the induced subarrangement of every three pseudocircles
only has triangular faces.  

\paragraph*{Results and outline}

In Section~\ref{sec:3colarr} we discuss two infinite families of
3-colorable arrangements.  

The first is the family of
$\triangle$-saturated arrangements of pseudocircles: A plane graph is \emph{$\triangle$-saturated} if  every edge is incident to exactly one triangular face, an arrangement is \emph{$\triangle$-saturated} if its arrangement graph is $\triangle$-saturated.
The second family is based on a specific construction which replaces a
pseudocircle by a bundle of three pseudocircles and preserves 3-colorability.

In Section~\ref{sec:corona} we extend our study of $\triangle$-saturated
arrangements and present an infinite family of arrangements which require $4$
colors. We believe that the construction results
in infinitely many $4$-vertex-critical arrangement graphs. 
A $k$-chromatic graph is \emph{$k$-vertex-critical} 
if the removal of every vertex decreases the chromatic number.  
It is \emph{$k$-edge-critical} if the removal of every edge decreases the chromatic number.
One of the arrangements which can be obtained with our construction is
Koester's arrangement of 7 circles~\cite{Koester1985}; see
Figure~\ref{fig:Koester2} in Section~\ref{subsec:corona}. 
Koester obtained his example using a ``crowning'' operation, which actually yields infinite
families of $4$-edge-critical $4$-regular planar graphs.  However, except for
the initial 7 circles example, these graphs are not arrangement graphs of arrangements of pseudocircles.

In Section~\ref{sec:fractional} we investigate the fractional chromatic number
$\chi_f$ of arrangement graphs. Roughly speaking, this variant of the chromatic number is the
objective value of the linear relaxation of the ILP formulation for the
chromatic number\footnote{The exact definition of the fractional chromatic number is deferred to Section~\ref{sec:fractional}}.  
We show that intersecting arrangements of pseudocircles are
``close'' to being $3$-colorable by proving that
$\chi_f(\mathcal{A}) \le 3+O(\frac{1}{n})$ 
for any intersecting arrangement $\mathcal{A}$ of $n$ pseudocircles.

In their work about the fractional chromatic number of planar graphs, 
Gimbel, K\"{u}ndgen, Li, and Thomassen conjectured that every $4$-chromatic planar
graph has fractional chromatic number strictly greater
than~$3$~\cite[Conjecture~3.2]{GimbelKLT2019}. They argued that a positive
answer to this statement would yield an alternative proof of the Four Color
Theorem.  In Section~\ref{sec:gimbel}, we present an example of a
$4$-edge-critical arrangement graph which is fractionally $3$-colorable.  The
example is the basis for constructing an infinite family of $4$-regular planar
graphs which are $4$-edge-critical and fractionally $3$-colorable.  This
disproves the conjecture of Gimbel et al.~in a strong form.

We conclude this paper with a discussion in Section~\ref{sec:discussion}, where we also propose three 
strengthened versions of Conjecture~\ref{conj:FHNS} which are
supported by exhaustive experiments with small arrangements.

\section{Families of 3-colorable arrangements of pseudocircles}
\label{sec:3colarr}

In this section we present two classes of arrangements of pseudocircles which
are 3-colorable.

\subsection{$\triangle$-saturated arrangements are 3-colorable}
\label{sec:trianglesat}

Recall that an arrangement is \emph{$\triangle$-saturated} if every edge of
the arrangement graph is incident to exactly one triangular face.  Figure~\ref{fig:trianglesat} shows
some examples of $\triangle$-saturated arrangements of pseudocircles.
We show that $\triangle$-saturated arrangements are 3-colorable.
This verifies Conjecture~\ref{conj:FHNS} for a class of
great-pseudocircle arrangements. 

Note, however that not all $\triangle$-saturated arrangements are great-pseudocircle arrangements;
For example the first two arrangements in Figure~\ref{fig:trianglesat} are not.
To see this, consider the subarrangement of the black, blue, and red pseudocircle in each of the two arrangements.

\begin{figure}[!htb]
	\centering
	\hbox{}
	\hskip3mm
	\begin{subfigure}[b]{.29\textwidth}
		\centering
		\includegraphics[page=1,width=\textwidth]{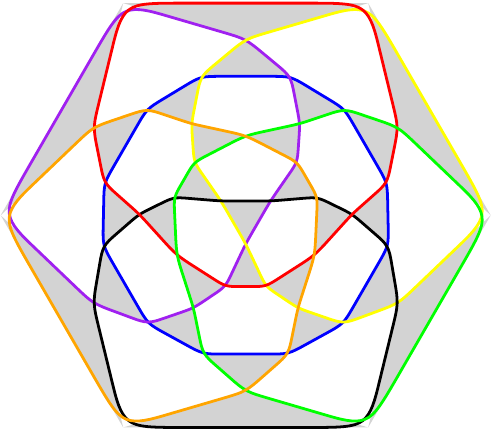}
	\end{subfigure}
	\hfill
	\begin{subfigure}[b]{.26\textwidth}
		\centering
		\includegraphics[page=1,width=\textwidth]{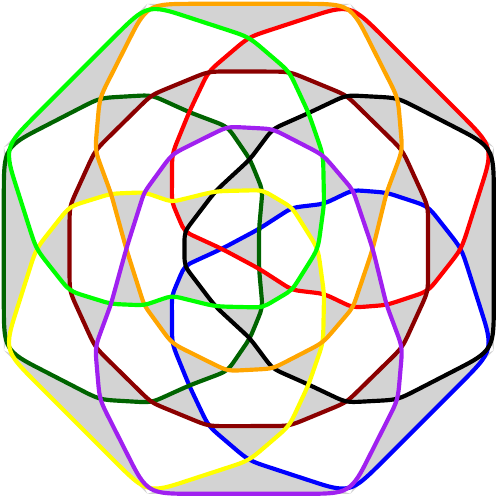}
	\end{subfigure}
	\hfill
	\begin{subfigure}[b]{.28\textwidth}
		\centering
		\includegraphics[page=1,width=\textwidth]{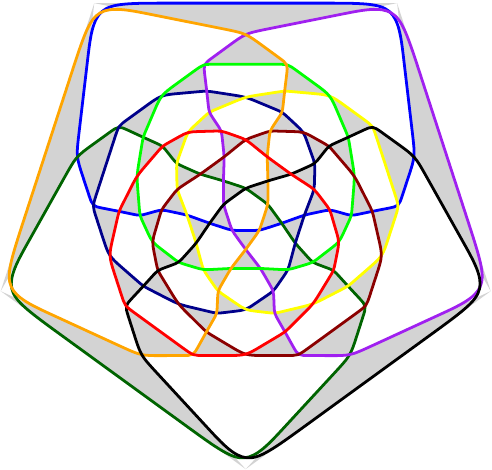}
	\end{subfigure}
	\hskip3mm
	\hbox{}
	\caption{$\triangle$-saturated intersecting  arrangements with 7, 9, and 10 pseudocircles.
	}
	\label{fig:trianglesat}
\end{figure}

\begin{theorem}
	\label{thm:3_colorable_arr}
	Every simple $\triangle$-saturated  arrangement $\AA$ of pseudocircles is $3$-colorable.
\end{theorem}

\begin{proof}
  Let $H$ be a graph whose vertices correspond to the triangles of $\AA$ and
  whose edges correspond to pairs of triangles sharing a vertex of $\AA$.
  This graph $H$ is planar and 3-regular.  Moreover, since the arrangement
  graph of $\AA$ is 2-connected, $H$ is bridgeless.  Now Tait's theorem, a
  well known equivalent of the 4-color theorem, asserts that $H$ is
  3-edge-colorable, see e.g.~\cite{Aig87} or~\cite{Thomas98}. The edges of $H$
  correspond bijectively to the vertices of the arrangement~$\AA$ and, since
  adjacent vertices of $\AA$ are incident to a common triangle, the
  corresponding edges of $H$ share a vertex.  This shows that the graph of
  $\AA$ is $3$-colorable.
\end{proof}

The maximum number of triangles in arrangements of pseudolines and
pseudocircles has been studied intensively, see
for example \cite{Gruenbaum1972,Roudneff1986,Blanc2011} and the recent work
\cite{FelsnerScheucher2020}.  

By recursively applying the ``doubling method'',
Harborth \cite{Harborth1985}, Roudneff \cite{Roudneff1986}, and Blanc
\cite{Blanc2011} proved the existence of $\triangle$-saturated arrangements of
$n$ pseudolines for infinitely many values of $n \equiv 0,4 \pmod 6$.
Similarly, a doubling construction for arrangements of
\mbox{(great-)}pseudocircles yields infinitely many $\triangle$-saturated
arrangements of \mbox{(great-)}pseudocircles.  Figure~\ref{fig:doubling}
illustrates the doubling method applied to an arrangement of
great-pseudocircles.  Note that for $n \equiv 2 \pmod 3$ there is no
$\triangle$-saturated intersecting pseudocircle arrangements because the
number of edges of the arrangement graph equals 3 times the number of
triangles but the number of edges is $2n(n-1)$ which is not divisible by~3.

\begin{figure}[htb]
	\hbox{}
	\hskip7mm
	\begin{subfigure}[b]{.42\textwidth}
		\centering
		\includegraphics[page=1,width=\textwidth]{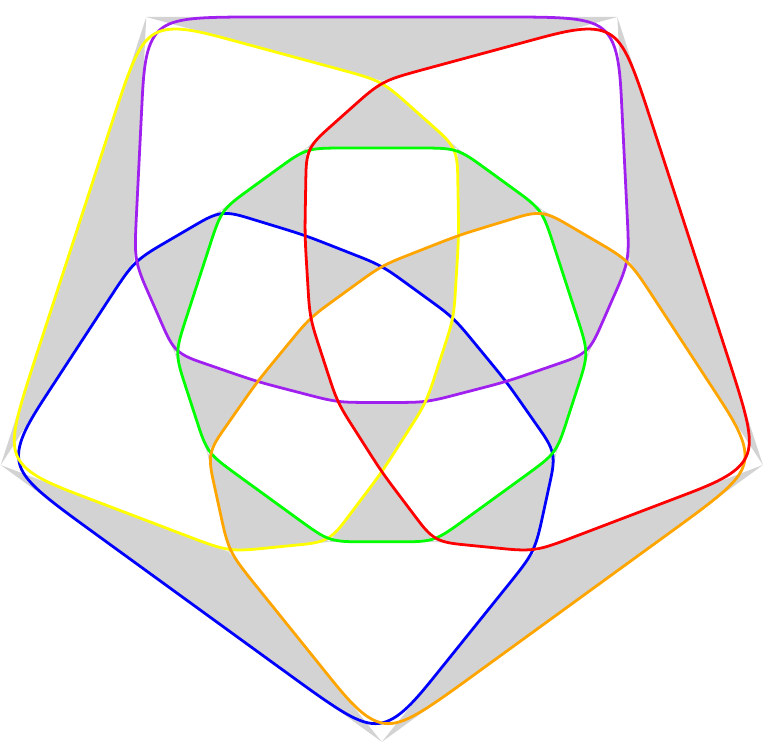}
		\caption{}
		\label{fig:doubling1}
	\end{subfigure}
	\hfill
	\begin{subfigure}[b]{.42\textwidth}
		\centering
		\includegraphics[page=2,width=\textwidth]{figs/Harborth_doubling}
		\caption{}
		\label{fig:doubling2}
	\end{subfigure}	
	\hskip7mm
	\hbox{}
	\caption{ The doubling method applied to an arrangements of 6
          great-pseudocircles. The red pseudocircle is replaced by a cyclic
          arrangement.}
	\label{fig:doubling}
\end{figure}

The proof of Theorem~\ref{thm:3_colorable_arr} actually can be extended to a
larger class of graphs (cf.\ Theorem~\ref{thm:3_colorable_gr}). Before stating
the result we need some more definitions.

The \emph{medial graph} $M(G)$ of an embedded planar graph $G$ is a graph
representing the adjacencies between edges in the cyclic order of vertices and
faces, respectively: The vertices of $M(G)$ correspond to the edges in $G$.
Two vertices of $M(G)$ share an edge whenever their corresponding edges in $G$ are
adjacent along the boundary of a face of $G$ (and hence consecutive around a
vertex; vertices of degree 1 and 2 in $G$ induce loops and multi-edges, respectively, in $M(G)$).
Note that every medial graph is a 4-regular planar graph. 
Vice-versa, every 4-regular planar graph is the medial graph of some planar graph. 

In order to see that the latter statement is true for connected graphs, let $H$ be a 
$4$-regular connected
embedded planar graph, and consider its dual graph $H^\ast$.
Since $H$ is $4$-regular and hence 
Eulerian, $H^\ast$ is a bipartite graph.
Next consider the $2$-coloring of the
faces of~$H$ which is induced by the bipartition of $H^\ast$, say, with colors
gray and white. Pick one of the color classes, e.g., the gray faces, and
create a new plane graph $G$ as follows: $G$ has exactly one vertex placed in
the interior of every gray face of $H$, and two vertices $u$ and $v$ of~$G$
are connected via an edge if and only if their corresponding gray faces touch
at a vertex~$x$. In this case, the edge $uv$ is drawn in $G$ in such a way
that it connects $u$ to $v$ by crossing through~$x$ and staying within the
union of the gray faces corresponding to $u$ and $v$ otherwise. From this
construction it is now easy to see that $G$ is a plane graph satisfying
$M(G)=H$, and every such graph $G$ is referred to as a \emph{premedial graph} of
$H$. By picking the white instead of the gray faces in the above construction,
we would have obtained another premedial graph of $H$, namely the dual graph
$G^\ast$ of $G$. While this shows that reconstructing $G$ from $M(G)$ is in
general not possible, it can be seen that $H$ determines a primal-dual pair
$\{G,G^*\}$ of premedial graphs uniquely up to isomorphism.
Figure~\ref{fig:premedial} shows an example of a medial graph and its two premedial graphs.

% in einem figure environment mit caption
   \calc_figscale{20}
    \begin{figure}[htb]
    \centerline{\input{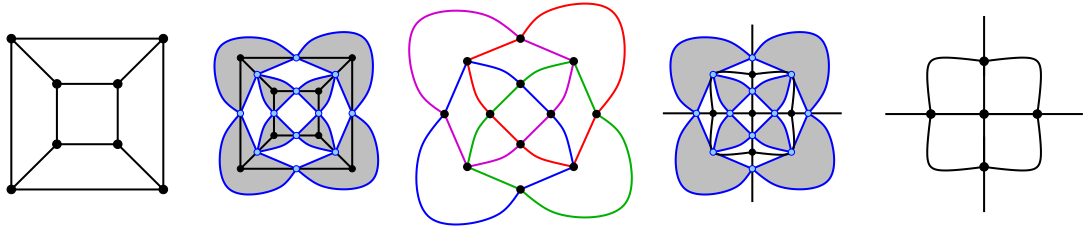}}
    \caption{The cube graph (left) and its medial graph $G_\AA$ (middle). The graph
	$G_\AA$ is also the graph of an arrangement of four pseudocircles; as indicated by the edge colors. 
	The second premedial graph of $G_\AA$ is the octahedron graph (right).\label{fig:premedial}}
    \end{figure}

Note that if $G$ is the graph of an arrangement of pseudocircles, then $G$ is 
4-regular while its dual graph $G^*$ has vertices of degree $\leq 3$. Hence, in this
case we can identify $G$ in the pair of premedial graphs given by $M(G)$.

In the other direction, an arrangement graph $G$ has a cubic premedial graph 
-- the graph $H$ in the proof of Theorem~\ref{thm:3_colorable_arr} --
if and only if $G$ is $\Delta$-saturated.
Moreover, the proof of Theorem~\ref{thm:3_colorable_arr} does not require that the 4-regular graph $G$ 
under consideration is actually an arrangement graph. It just requires it to be 2-connected to ensure that the cubic premedial graph $H$ is bridgeless. 
Hence the following theorem generalizes Theorem~\ref{thm:3_colorable_arr}, while essentially having the same proof.

\begin{theorem}\label{thm:3_colorable_gr}
  If $G$ is a 2-connected 4-regular planar graph which has a cubic premedial
  graph $H$ then $\chi(G)=3$.
\end{theorem}

We remark that 2-connectivity is a crucial condition in
Theorem~\ref{thm:3_colorable_gr} as illustrated in
Figure~\ref{fig:1connected_premedial}.

\begin{figure}[htb]
	\centering
	\includegraphics[width=0.3\textwidth]{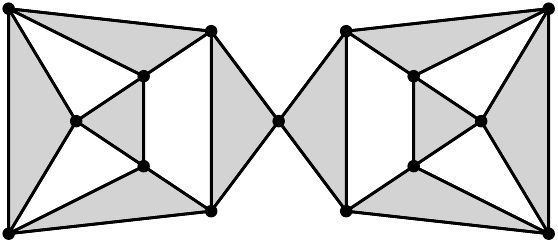}
	\caption{A connected 4-regular planar graph $G$ with a cubic premedial graph and $\chi(G)=4$}
	\label{fig:1connected_premedial}
\end{figure}

\begin{proof}
  Let $H$ be the cubic premedial graph of~$G$, i.e., $G=M(H)$. A bridge in~$H$
  corresponds to a cut vertex of its medial graph~$G$. Since $G$ is assumed to
  be 2-connected, it follows that $H$ is bridgeless and hence, by Tait's
  theorem, 3-edge-colorable. Adjacent vertices of $G$ correspond to edges
  of~$H$ that are consecutive in the circular order at a vertex of~$H$.  As
  such pairs of edges receive different colors in the edge-coloring of~$H$,
  the 3-edge-coloring of $H$ induces a 3-coloring of~$G$.
\end{proof}

It follows from the above discussion that $\chi(M(H))$ is upper bounded
by the \emph{chromatic index}~$\chi'(H)$, i.e., the minimum number
of colors required for a proper edge coloring of~$H$. Indeed, if $v$ is a vertex
of $H$, then edges incident to $v$ require pairwise distinct colors
in an edge coloring, while in $M(H)$ these edges are vertices along the boundary of a facial cycle
so that repetitions of colors might be feasible.

\subsection{More families of 3-colorable arrangements}
\label{sec:another_construction}

We next show how to construct more infinite families of 3-colorable arangements of 
(intersecting) pseudocircles, great-pseudocircles, or circles, respectively.

Let $\AA$ be a 3-colorable arrangement of $n$ pseudocircles
and let $\phi$ be a coloring of $\AA$ with colors $0,1,2$.
We will use the additive structure of $\ZZ_3$ on the colors.

Fix a pseudocircle $C$ of~$\AA$ and let $V_I$ and $V_O$ be the sets of vertices
of $\AA$ inside and outside of~$C$, respectively.  Let $\AA'$ be the
arrangement obtained from $\AA$ by adding two parallel pseudocircles~$C'$ and
$C''$ along $C$, i.e., the order in which the three pseudocircles $C$, $C'$, and
$C''$ cross the other pseudocircles is the same. We can think of the parallel pseudocircles
as drawn close to $C$ such that~$C$ is the innermost, $C'$ the middle, and
$C''$ the outer of the three pseudocircles. For every vertex $v\in C$, we have the
corresponding vertices $v'$ and $v''$ on $C'$ and $C''$ respectively.
Formally, this correspondence can be stated by saying that $vv'$ and $v'v''$ are edges of $\AA'$, and edges $vw$ with $w\in V_O$ of $\AA$ are replaced by $v''w$ in $\AA'$.

The following defines a 3-coloring $\phi'$ of $\AA'$: 
For $u \in V_I$ let
$\phi'(u) = \phi(u)$; for a triple $v,v',v''$ of corresponding vertices on the
three pseudocircles $C,C',C''$, let $\phi'(v) = \phi(v)$,  $\phi'(v') = \phi(v)+1$, and
$\phi'(v'') = \phi(v)+2$; finally for $w \in V_O$ let $\phi'(w) = \phi(w)+2$.

Since we are mostly interested in intersecting arrangements we next describe
how to transform $\AA'$ into a 3-colorable intersecting arrangement $\AA''$.
Let $e_1$ and $e_2$ be two edges on $C$ in~$\AA$. Corresponding to each of
$e_1$ and $e_2$, we have a  $2\times 3$ grid in $\AA'$; see Figure~\ref{fig:makeX} left. 
This grid can be replaced by a triangular structure with pairwise
crossings of the three pseudocircles, see Figure~\ref{fig:makeX} middle and
right. 
The figure also shows that a 3-coloring of the grid, where the
colors in the columns are $0,1,2$, or $1,2,0$, or $2,0,1$, can be extended to
the three added crossings.  Hence, we obtain a 3-colorable intersecting
arrangement $\AA''$.

% in einem figure environment mit caption
   \calc_figscale{30}
    \begin{figure}[htb]
    \centerline{\input{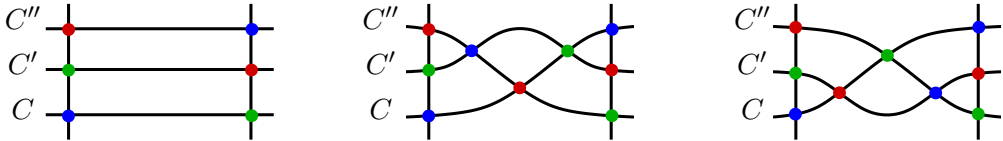}}
    \caption{A $2 \times 3$ grid (left) and two ways of adding pairwise
  crossings on the horizontal curves.\label{fig:makeX}}
    \end{figure}

Let $\AA$ be a 3-colorable arrangement of great-pseudocircles.  If we pick
$e_1$ and $e_2$ as a pair of antipodal edges on $C$ and add the intersections
between $C$, $C'$, and $C''$ along those two edges, once as in the middle of Figure~\ref{fig:makeX} and once as in the right of the figure, then we obtain an arrangement $\AA''$ which is again an arrangement of great-pseudocircles.

Moreover, if $\AA$ is an arrangement of (proper) circles, then clearly $\AA'$ is again an arrangement of circles. Less obvious but still true is that $\AA''$ can also
be realized as a circle arrangement. The reason is that the three circles $C',C''$ can be placed inside an arbitrarily narrow belt centered at $C$.
Figure~\ref{fig:Adoubleprime} shows an
example of a transformation $\AA \to \AA' \to \AA''$.

% in einem figure environment mit caption
   \calc_figscale{13}
    \begin{figure}[htb]
    \centerline{\input{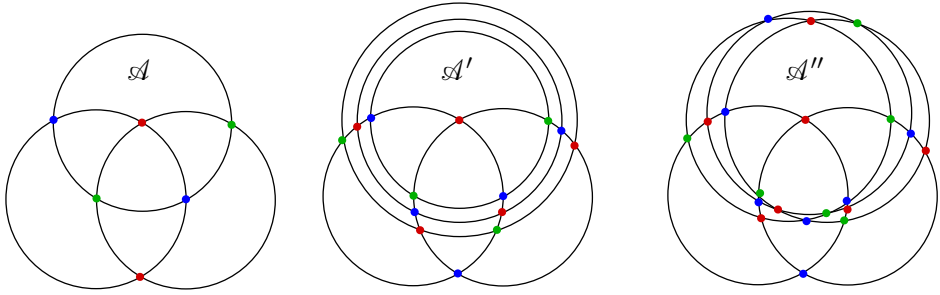}}
    \caption{A 3-colorable arrangement $\AA$ of circles and
  the derived arrangements $\AA'$ and $\AA''$.\label{fig:Adoubleprime}}
    \end{figure}

 The following is a direct consequence of the above-described constructions. 
 \begin{proposition}
  Let $\AA$ be a 3-colorable arrangement of $n$ (intersecting) pseudocircles, great-pseudocircles, or circles, respectively.
  Then for any $k\in \mathbb{N}$, arrangement $\AA$ can be extended to a 3-colorable arrangement of size $n+2k$ of the same type.
 \end{proposition}

\section{Constructing 4-chromatic arrangement graphs}
\label{sec:corona}

In the first part of this section, we describe an operation that extends any
$\triangle$-saturated intersecting arrangement of pseudocircles with a
pentagonal cell (which is $3$-colorable by Theorem~\ref{thm:3_colorable_arr})
to a 4-chromatic arrangement of pseudocircles by inserting only one additional
pseudocircle. This corona extension is somewhat related to Koester's crowning,
an operation used to construct an infinite family of 4-regular 4-edge-critical 
planar graphs \cite{Koester1990}. This motivates the study of criticality of the graphs obtained via
the corona extension, which is the topic of Subsection~\ref{subsec:criticality}.

\subsection{The corona extension}
\label{subsec:corona}

We start with a $\triangle$-saturated arrangement $\AA$ of pseudocircles which contains a pentagonal cell $\pentagon$.  
By definition, in the 2-coloring of the faces of $\AA$, one of the two color classes consists of triangles only; 
see e.g.\ the arrangement from Figure~\ref{fig:Koester1}.  
Since the arrangement is $\triangle$-saturated, the pentagonal cell $\pentagon$ is surrounded by triangular cells.  

We can now insert an additional pseudocircle enclosing $\pentagon$ so that the new
pseudocircle intersects only the 5 pseudocircles which bound $\pentagon$ and does so only at edges incident to vertices of $\pentagon$.
Figure~\ref{fig:Koester2} illustrates this extension for the arrangement from Figure~\ref{fig:Koester1}.

In the extended arrangement $\AA^+$, one of the two color classes of faces
consists of triangles and the pentagon~$\pentagon$.  We say that $\AA^+$ is
obtained via a \emph{corona extension}\footnote{
  The writing of this article has benefited from the corona lockdown in April
  2020.}  from $\AA$.  It is interesting to note that the arrangement depicted
in Figure~\ref{fig:Koester2} is Koester's arrangement
\cite{Koester1985}.

\begin{figure}[htb]

	\hbox{}
	\hfill
	\begin{subfigure}[b]{.33\textwidth}
		\centering
		\includegraphics[page=1,width=\textwidth]{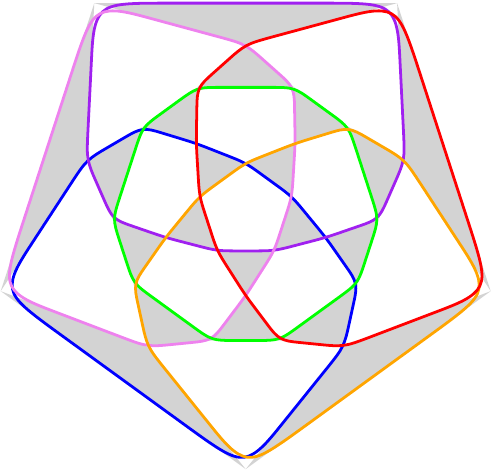}
		\caption{}
		\label{fig:Koester1}
	\end{subfigure}
	\hfill
	\begin{subfigure}[b]{.33\textwidth}
		\centering
		\includegraphics[page=2,width=\textwidth]{figs/koester_corona}
		\caption{}
		\label{fig:Koester2}
	\end{subfigure}
	\hfill
	\hbox{}
	
	\caption{ \subref{fig:Koester1}~A $\triangle$-saturated arrangement of
          6 great-circles and \subref{fig:Koester2}~the corona extension at its
          central pentagonal face.  The arrangement in \subref{fig:Koester2} is
          Koester's~\cite{Koester1985} example of a planar 4-edge-critical
          4-regular planar graph.}
	\label{fig:Koester}
\end{figure}

\goodbreak

To discuss the colorability of the corona extension, we introduce some
notation.  For a graph~$G$, let $\alpha(G)$ denote the size of any maximum
independent set of $G$.  In a proper $k$-coloring of $G$, the vertices of
every color class form an independent set, and we trivially have
$\alpha(G) \geq \frac{|V(G)|}{k}$ for every $k$-colorable graph.

\begin{lemma}\label{theorem:alpha_lt_V3}
  Let $G$ be a 4-regular planar graph. If in the 2-coloring of the faces of
  $G$, one of the classes consists of only triangles and a single pentagon,
  then $\alpha(G)<\frac{|V(G)|}{3}$.
\end{lemma}

\begin{proof}
  Color the faces of $G=(V,E)$ with black and white. Let the black class
  contain only triangles and one pentagon. Let $t$ be the number of
  these triangles and let $\alpha:=\alpha(G)$.  Given an independent set $I$
  of cardinality $\alpha$, we count the number of pairs $(v,F)$, where $v$ is
  a vertex of $I$ and $F$ is a black face of $\mathcal{A}$ incident to
  $v$. There are 2 such faces for every $v\in I$, hence, $2\alpha$ pairs in
  total.  Since any independent set of $G$ contains at most one vertex of each
  triangle and at most two vertices of the pentagon, the number of pairs
  $(v,F)$ is at most $t + 2$. Hence, we have
  \begin{equation}
    \label{eqn:2alpha_le_triangle_plus_2}
    2\alpha\le t + 2.
  \end{equation}
  Since $G$ is 4-regular, there are exactly $|E|=2|V|$ edges.  As every edge
  is incident to exactly one black face, we also have $|E|=3 t + 5$.
  This yields the equation
  \begin{equation}
    \label{eqn:3triangle_plus_5_is_2V}
    3 t + 5 = 2|V|.
  \end{equation}
  From equation~(\ref{eqn:3triangle_plus_5_is_2V}), we conclude that
  $t$ is odd. Therefore we can strengthen
  equation~(\ref{eqn:2alpha_le_triangle_plus_2}) to
  \begin{equation}
    \label{eqn:2alpha_le_triangle_plus_1}
    2\alpha\le t + 1.
  \end{equation}
  Combining equations~(\ref{eqn:3triangle_plus_5_is_2V})
  and~(\ref{eqn:2alpha_le_triangle_plus_1}) yields
  $6\alpha\le 3 t + 3 = 2|V| -2 $ and hence $\alpha < \frac{|V|}{3}$.
\end{proof}

\begin{proposition}\label{prop:corona_4col}
  The corona extension of a $\triangle$-saturated arrangement of pseudocircles
  with a pentagonal cell~$\pentagon$ is 4-chromatic.
\end{proposition}

\begin{proof}
  From Lemma~\ref{theorem:alpha_lt_V3} we know that after the corona extension
  the inequality $3\alpha(G) < |V(G)|$ holds.  This implies that the corona
  extension of a $\triangle$-saturated arrangement of pseudocircles with a
  pentagonal cell~$\pentagon$ is not 3-colorable.
\end{proof}

It is remarkable that the argument from the proof of
Lemma~\ref{theorem:alpha_lt_V3} only holds for pentagons. More precisely, if
the class of black faces of $G$ consists of triangles and a single $k$-gon,
then we need $k=5$ to get $\alpha < |V|/3$.

By iteratively applying the doubling method (cf.\
Section~\ref{sec:trianglesat}) to the arrangements depicted in
Figure~\ref{fig:trianglesat}, we obtain $\triangle$-saturated arrangements of
$n$ pseudocircles which have pentagonal cells for infinitely many values of
$n \equiv 0,4 \pmod 6$.  Applying the corona
extension to the members of this infinite family yields an infinite family
of arrangements that are not 3-colorable.

\begin{theorem}
\label{thm:pentagon}
There exists an infinite family of simple 4-chromatic arrangements of
pseudocircles, each of which is obtained from an intersecting arrangement of
pseudocircles by adding only one additional pseudocircle.
\end{theorem}

\subsection{Criticality}
\label{subsec:criticality}

Koester~\cite{Koester1990} introduced the \emph{crowning} operation and used
this operation to construct an infinite family of 
$4$-regular $4$-edge-critical planar graphs (cf. Proposition \ref{prop:koester_edge_critical} and Figure \ref{fig:krown}). A particular example of a graph obtained by crowning
is the Koester graph of Figure~\ref{fig:Koester2}, which happens to be an
arrangement graph of circles. 

Since crowning and the corona extension show some similarities and both
operations can be used to obtain the Koester graph depicted in Figure~\ref{fig:Koester2},
we believe that many of the 4-chromatic arrangements obtained with the 
corona extension (Theorem~\ref{thm:pentagon}) are in fact  
4-vertex-critical. In the following, we present
sufficient conditions to obtain 4-vertex-critical and 4-edge-critical arrangements via the corona
extension.

We need some terminology.  
Let $H$ be a cubic plane graph and let $G=M(H)$ be its medial graph.  If $H$ is
bridgeless, then $\chi(G)=3$ by Theorem~\ref{thm:3_colorable_gr}.  If in addition $H$ has
a pentagonal face $\pentagon_H$, then we can apply the corona extension to $G$ to obtain
a 4-regular graph $G^\circ$ with
$\chi(G^\circ)=4$ (Lemma~\ref{theorem:alpha_lt_V3}).  We are interested in conditions on $H$ which imply that~$G^\circ$ is 4-vertex-critical or even 4-edge-critical. 

With $\pentagon_G$ we denote the pentagon corresponding to $\pentagon_H$ in $G$. The \emph{connector
  vertices} of~$\pentagon_G$ are the five
vertices of the triangles adjacent to $\pentagon_G$ which do not belong to $\pentagon_G$.

% in einem figure environment mit caption
   \calc_figscale{18}
    \begin{figure}[htb]
    \centerline{\input{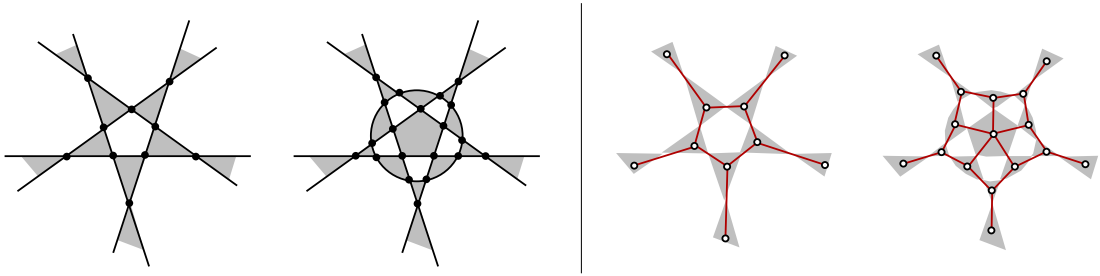}}
    \caption{Applying the corona extension at a pentagon of a $\triangle$-saturated 4-regular planar graph.  \label{fig:corona}}
    \end{figure}

Consider a 3-edge-coloring $\varphi$ of $H$. 
We call $\varphi$ \emph{trihamiltonian}, if all three subgraphs induced by edges of two of the three colors of~$\varphi$ induce a Hamiltonian cycle on $H$. We will prove the following:

\begin{theorem}\label{thm:ve-critical}
  Let $H$ be a cubic planar graph with a pentagonal face $\pentagon_H$ 
  and a trihamiltonian 3-edge-coloring~$\varphi$.  
  If $G$ is the medial graph of $H$ and $\GC$
  is obtained from $G$ by the corona extension at $\pentagon_G$, 
  where $\pentagon_G$ is the pentagonal face of $G$ corresponding to $\pentagon_H$, then $G^\circ$ is 4-vertex-critical. 
  If, additionally, $H$ admits 5-fold rotational symmetry around $\pentagon_H$, 
  then $G^\circ$ is even 4-edge-critical.
\end{theorem}

\begin{figure}[!htb]
  \begin{subfigure}[b]{.2\textwidth}
    \centering
    \includegraphics[page=1, width=\textwidth]{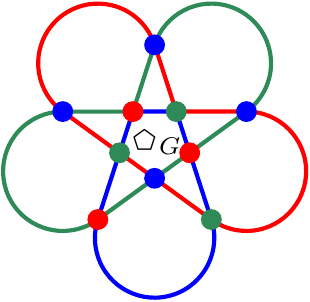}
    \label{fig:critical_edges1}
    \vspace{68pt}
  \end{subfigure}
  \hfill
  \begin{subfigure}[b]{.5\textwidth}
    \centering \includegraphics[page=6,width=\textwidth]{criticalityHamCycl.pdf}
    \label{fig:critical_edges2}
  \end{subfigure}
  \hfill
  \begin{subfigure}[b]{.25\textwidth}
    \centering \includegraphics[page=10,width=\textwidth]{criticalityHamCycl.pdf}
    \includegraphics[page=11,width=\textwidth]{criticalityHamCycl.pdf}
    \label{fig:critical_edges3}
  \end{subfigure}
  
  \caption{Left: Edges of $G$ have the color that is missing on its incident vertices.\\
  Middle: The golden vertex can move along the red path and arrows in both directions. \\
  Right: The purple edge indicates the critical edge after we change colors on the green cycle.}
  \label{fig:critical}
\end{figure}

\begin{proof}
  Recall from Section \ref{sec:trianglesat} that the 3-edge-coloring $\varphi$ of $H$ yields a 3-vertex-coloring of the $\triangle$-saturated graph $G$. 
  Each of the three 2-colored Hamiltonian cycles in $H$ given by $\varphi$ yields a cycle in $G$, which covers all the vertices of the respective colors. This is indicated in Figure \ref{fig:critical} (left). 
  Each edge of $G$ is contained in exactly one of the 3 cycles, hence we obtain a non-proper 3-edge-coloring of $G$ with the property that every color class is a cycle. The two red, two green and one blue circular arcs indicate the way that these three cycles are closed outside the corona region. Note that the order of connector vertices on the red and green cycle must be as indicated in Figure \ref{fig:critical} (left) since each monochromatic cycle is non-crossing. 
  Every vertex belongs to two of the three cycles induced by the edge coloring, hence, in Figure \ref{fig:critical} arcs of different colors can have multiple intersections and touchings. 
  
  Note that Figure \ref{fig:critical} (left) has a vertical axis of symmetry which preserves the blue vertices and the blue cycle but exchanges the colors red and green. In the following, we will show how to modify these two colorings (original and reflected) in order to find a collection of 4-colorings of $\GC$ that allows to argue for 4-vertex- and 4-edge-criticality in the respective cases.

  Figure \ref{fig:critical} (middle) shows a 4-coloring $\varphi^\circ$ of $\GC$ with the same coloring of the connector vertices and vertices outside the corona region as in Figure \ref{fig:critical} (left). Note that a single vertex is colored with the fourth color (gold). If in a 4-vertex-coloring $\varphi'$ a vertex is the single golden vertex, we call it the \emph{special vertex} of the coloring. To show that $\GC$ is 4-vertex-critical, we need to show that every vertex in the graph is the special vertex of some 4-coloring of $\GC$. 
  
  In every 4-coloring with a special vertex $v$, this vertex is surrounded by all three colors, since we know from Lemma \ref{theorem:alpha_lt_V3} that the graph $\GC$ is 4-chromatic. Thus only one of the colors appears twice. Recoloring $v$ with the one of the other colors and the corresponding neighbor $w$ of $v$ with gold makes this neighbor the special vertex of the new coloring. We say that the special vertex \emph{moves} from $v$ to $w$. To show that an edge $e$ is critical in $\GC$, it suffices to show that there is a 4-coloring with special vertex $v$ which allows such a move from $v$ to $w$.
  
  Starting from $\varphi^\circ$, we can make the special vertex move along the red path (see Figure \ref{fig:critical} (middle)), changing the colors of green and blue vertices along the way. To see this, remember that the green and blue vertices on the red arcs have 2 red neighbors (in $G$ and thus in $\GC$), so the blue and green color are the ones to move along. At one of the endpoints of the red path in Figure \ref{fig:critical} (middle), there are two blue neighbors, thus the next neighbor to move to is the red neighbor. At the other endpoint, there are two green neighbors, so again the red neighbor is the next neighbor to move to. This is indicated by arrows in Figure \ref{fig:critical} (middle).
  
 \begin{figure}[t]
  \centering
  
  \includegraphics{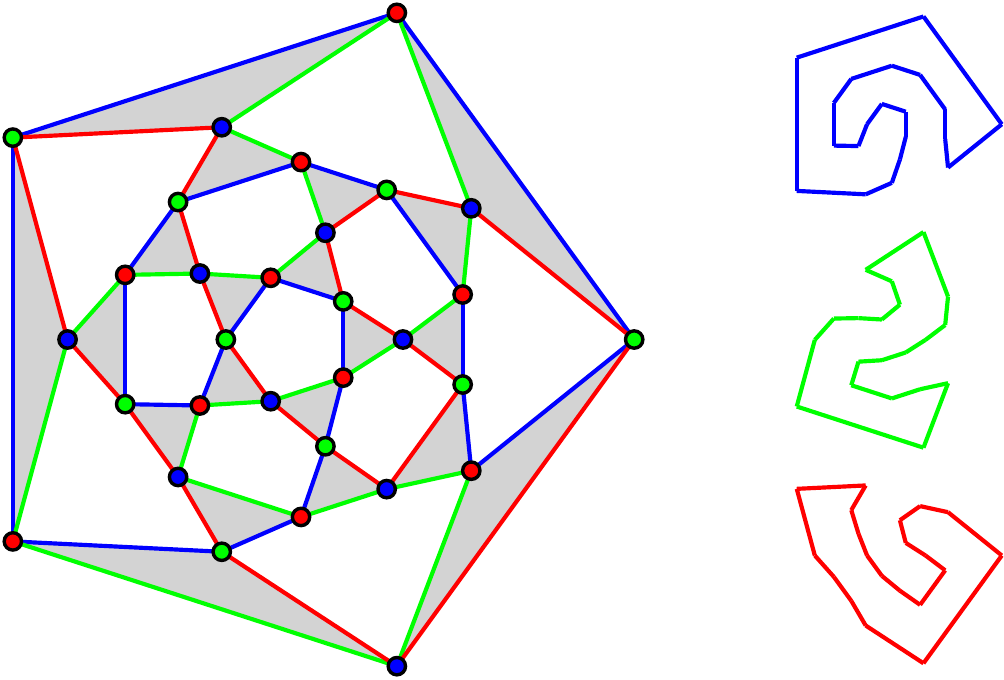}
  
  \caption{A 3-coloring of the great-circle arrangement from
    Figure~\ref{fig:Koester1}. 
    The three cycles obtained by removing each of the color classes are depicted on the right.}
  \label{fig:cycle3-koester}
 \end{figure}
 
\begin{figure}[p]
  \centering
  
  \includegraphics[width=0.7\textwidth]{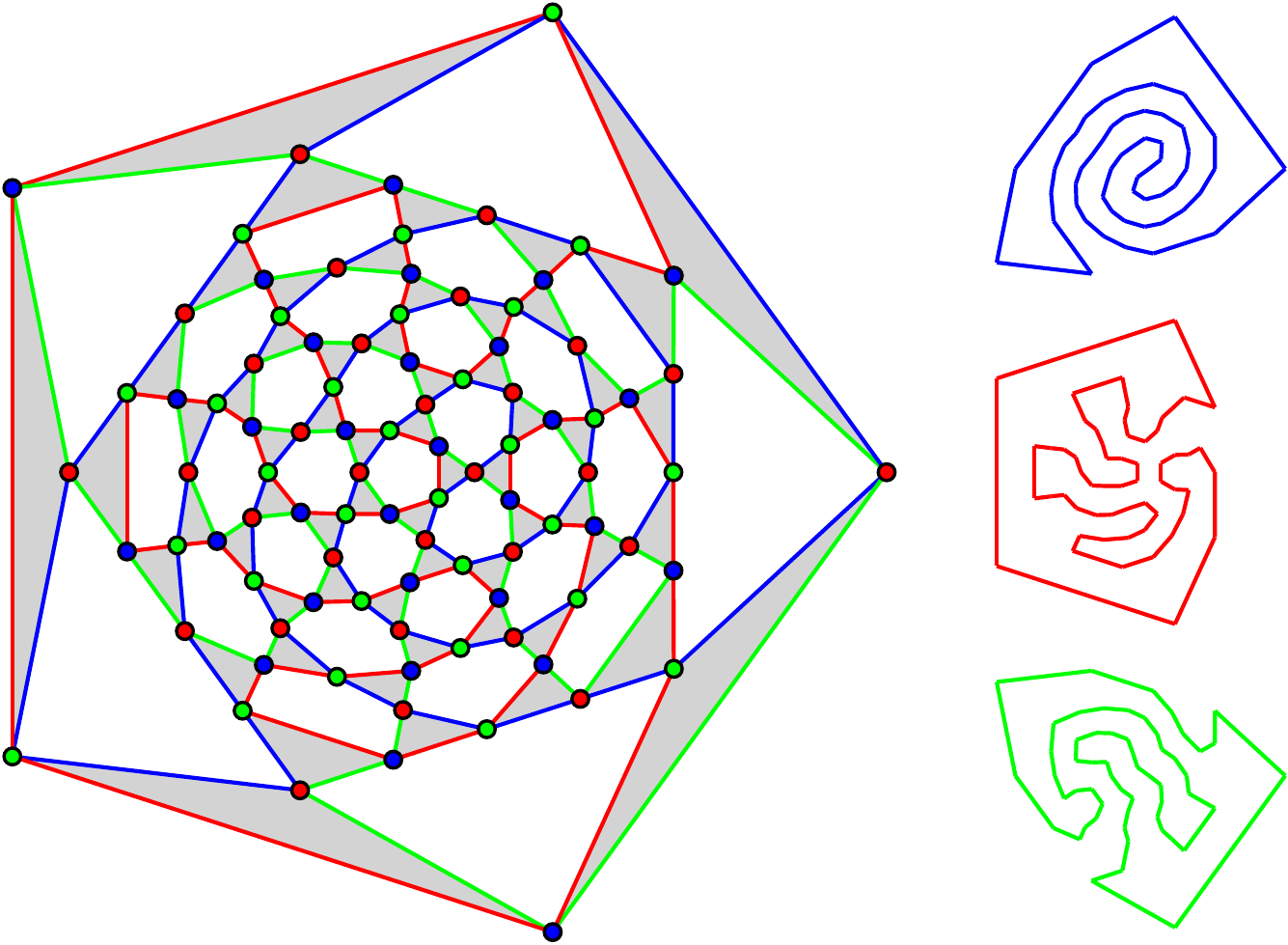}
  
  \caption{A 3-coloring of a $\triangle$-saturated great-circle arrangement.
    The three cycles obtained by removing each of the color classes are depicted on the right.}
  \label{fig:cycle3-n10}
  \centering
  
  \includegraphics{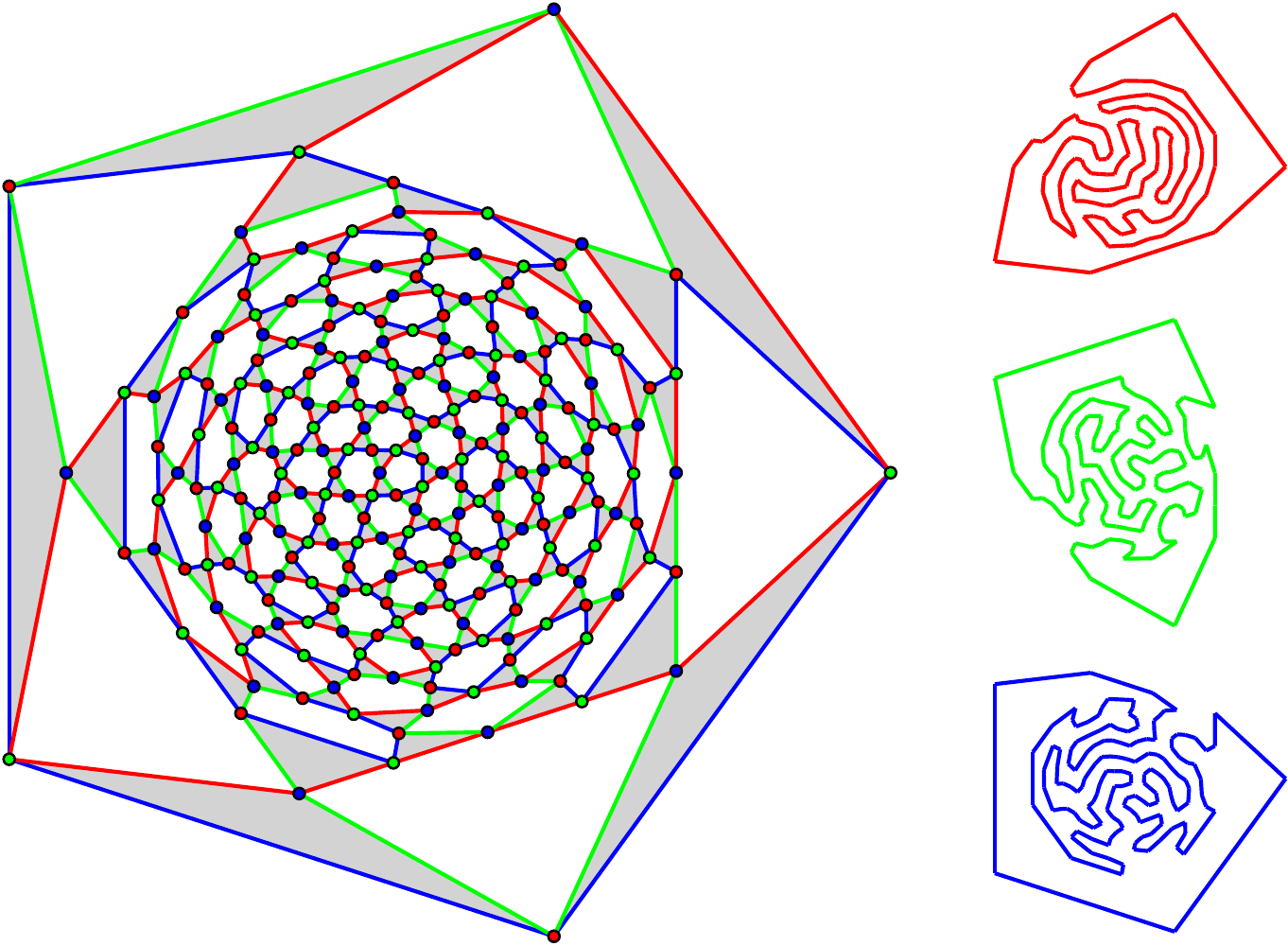}
  
  \caption{This great-circle arrangement of 16 great-pseudocircles has been discovered by Simmons \cite{Simmons73}.
    The three cycles obtained by removing any of the color classes are depicted on the right.}
  \label{fig:cycle3-n16}
\end{figure}
  Move the golden vertex along the two red branches and the extending steps. Then turn to the symmetric (reflected) coloring and do the symmetric moves. We claim that together this yields a collection of colorings of $\GC$ such that every vertex is the special (golden) vertex of one of them.
  For the vertices inside or on the new circle, this is easily checked from Figure \ref{fig:critical} (middle). The vertices outside of the corona region are colored with 3 colors. The blue and green ones lie on the red arcs and are therefore reached when moving the golden vertex along the red arcs. The red ones lie on the green paths and will therefore be reached if we start from the reflected coloring, because then these same vertices would be green and lie on the corresponding red arcs. Thus 4-vertex-criticality is established.
  
  Now suppose that $H$ has a 5-fold symmetry fixing $\pentagon_H$. This symmetry carries over to $G$ and $\GC$. Thus it is sufficient to show that any edge can be rotated to an edge that we have covered already. The only edges which are not covered by the moves along the extended red paths of Figure \ref{fig:critical} (middle) and its rotations are the small edges on the new circle that are inside the triangles of $G$ next to $\pentagon_G$. In Figure \ref{fig:critical} (right), we show an additional extension of the coloring of the connector vertices to the interior. Exchanging the colors of the blue and red vertices of the green cycle in the bottom left makes it possible to 3-color the graph by making the special vertex blue, if the purple edge is omitted. Since the purple edge is a representative of the last rotational orbit we did not cover yet, this yields 4-edge-criticality.
\end{proof}

Next, we present some examples of the application of Theorem \ref{thm:ve-critical}. Let $H$ be a cubic planar
graph which has a unique 3-edge-coloring up to permutations of the colors. Then this coloring is 
trihamiltonian, since if a graph induced by two colors has more than one component, we can change the two colors on this component alone and construct a different coloring. Thus for any pentagon in $H$, the resulting graph $\GC$ is 4-vertex-critical.

The class of uniquely 3-edge-colorable cubic graphs is well understood.
Fowler~\cite[Theorem 2.8.5]{FowlerPHD} characterized them as the graphs that can be obtained from
$K_4$ by successively replacing a vertex by a triangle. These are the duals of stacked triangulations, 
which are the uniquely 4-colorable planar graphs~\cite[Conjecture 1.2.1]{FowlerPHD}\footnote{This theorem in the thesis is called a conjecture, since it is first proved to be equivalent to the Fiorini-Wilson-Fisk Conjecture, which is proved much later as the main result of the thesis.}.

Additionally, Figures~\ref{fig:cycle3-koester}, \ref{fig:cycle3-n10},
and~\ref{fig:cycle3-n16} show $\triangle$-saturated arrangements of 6, 10, and 16 great-pseudocircles respectively, that admit 5-fold rotational symmetry. The arrangement graphs are shown with 3-colorings which correspond to trihamiltonian 3-edge-colorings of their respective premedial graphs. The theorem implies that the corona extension at the outer pentagon of these arrangements yields 4-edge-critical graphs. We are aware of three more \mbox{$\triangle$-saturated} arrangements of 6, 7, and 9 pseudocircles respectively, which have 4-edge-critical corona extensions. For these arrangements, however, the 4-edge-criticality is not implied by our theorem.
 All data is available on the supplementary website~\cite{pseudocircles_website}.

We conclude this section with the following conjecture:

\begin{conjecture}\label{conj:pentagon}
  There exists an infinite family of simple arrangement graphs of $4$-edge-critical arrangements of pseudocircles.
\end{conjecture}

Relaxing the condition of the conjecture to 4-regular planar graphs, this is a known result of Koester (see Proposition  \ref{prop:koester_edge_critical}).

\section{Fractional colorings}
\label{sec:fractional}

In this section, we investigate fractional colorings of arrangements. A
\emph{$b$-fold coloring} of a graph~$G$ with $m$ colors is an assignment of a
set of $b$ colors from $\{1,\ldots,m\}$ to each vertex of~$G$ such that the
color sets of any two adjacent vertices are disjoint.  The \emph{$b$-chromatic
  number}~$\chi_b(G)$ is the minimum $m$ such that $G$ admits a $b$-fold
coloring with $m$ colors. The \emph{fractional chromatic number} of $G$ is
$\chi_f(G) := {\displaystyle \lim_{b \to \infty}}\frac{\chi_b(G)}{b} =
{\displaystyle \inf_{b}} \frac{\chi_b(G)}{b}$.  With $\alpha(G)$ being the
independence number of $G$ and~$\omega(G)$ being the clique number of $G$, the following
inequalities hold:
\begin{equation}
\label{equation:alpha-chi-relation}
\max \left\{\frac{|V|}{\alpha(G)},\omega(G) \right\} \le \chi_f(G)
\le \frac{\chi_b(G)}{b} \le \chi(G).
\end{equation}

The fractional chromatic number forms a natural lower bound for the chromatic
number of graphs. While the chromatic number of quite some intersecting arrangements of
pseudocircles is four, at least their fractional chromatic number is always close to three:

\begin{theorem} \label{theorem:chifrac_3_2n}
  Let $G$ be the arrangement graph
  of a simple intersecting arrangement $\AA$ of $n$ pseudocircles, then
  $\chi_f (G)\le 3+\frac{6}{3n-2} = 3+\frac{2}{n}+o\left(\frac{1}{n}\right)$.  In
  particular, if $v$ denotes the number of vertices of $G$, then
  $\chi_f(G) \le 3+\frac{2}{\sqrt{v}}+o\left(\frac{1}{\sqrt{v}}\right)$.
\end{theorem}

\begin{proof}
  Fix an arbitrary circle $C \in \AA$ and let $V_C \subseteq V_G$ be the
  vertex set of $C$. Let $V_G\setminus V_C=V_I \cup V_O$, where $V_I$ and
  $V_O$ are the sets of vertices inside or outside of $C$, respectively.

\paragraph{Claim 1.}
The graphs $G[V_I]$ and $G[V_O]$ are $3$-colorable.
\begin{proof}
  We prove the claim for $G[V_I]$, the proof for $G[V_O]$ is analogous.  Let
  $C_0 \notin \AA$ be a tiny circle in some face of the arrangement in the
  interior of~$C$. The Sweeping Theorem of Snoeyink and Hershberger
  \cite[Theorem~3.1]{SnoeyinkHershberger1991} asserts that there exists a
  sweep which continuously transforms $C_0$ into $C$ such that at any time
  $\AA\cup C_0$ is an arrangement of pseudocircles.  Let $t = |V_I|$ and let
  $\pi=(v_1,v_2,\ldots,v_{t})$ be the ordering of the vertices of $V_I$
  induced by this sweep, i.e., $v_i$ for $i \in \{1,\ldots,t\}$ is the $i$-th
  vertex met by the sweep-pseudocircle $C_0$. Orient each edge of~$G[V_I]$ from the
  vertex of smaller index to the vertex of larger index. Note that on every
  pseudocircle $C'\in \AA$ this orientation induces at most two directed paths that
  share the starting point, the first vertex of $C'$ met by $C_0$.
  At every vertex $v \in V_I$ two pseudocircles cross and $v$ has at
  most one predecessor on each of the two pseudocircles (here we use the fact that
  $\AA$ is an intersecting arrangement, and hence every pseudocircle of $\AA$
  different from $C$ intersects both the interior and the exterior of
  $C$). Hence, in the acyclic orientation of $G$ defined above, every
  $v \in V_I$ satisfies $\indeg(v) \leq 2$. Thus, the greedy algorithm
  with the ordering $\pi$ yields a $3$-coloring of $G[V_I]$.
\end{proof}

Let us pause to note that just on the basis of this first claim we get
$\chi_f (G)\le 3+\frac{6}{n-2}$ which is not too far from the bound given in the
theorem. Indeed if for each pseudocircle $C$ of the arrangement we use 3 colors to
color $V\setminus V_C$, then every vertex receives $n-2$ colors, whence we
obtain a $b$-coloring with $b=n-2$ using $3n$ colors in total, i.e.,
$\chi_{n-2} (G)\le 3n$.

\paragraph{Claim 2.}
The graph $G[V_C]$ is $2$-colorable.
\begin{proof}
  Let $F$ be a face of the planar graph $G[V_C]$. Each vertex of $F$ is a
  crossing of $C$ with some $C'\neq C$ and each $C'\neq C$ contributes 0 or 2
  vertices to the boundary of $F$. This shows that every face of $G[V_C]$ is
  even whence $G[V_C]$ is a bipartite graph.
\end{proof}

\paragraph{Claim 3.}
For every weighting $w:V_G \rightarrow [0,\infty)$ there is an independent set
$I$ of $G$ such that $w(I) \ge (\frac{1}{3}-\frac{2}{9n})w(V_G)$.
\begin{proof}
  Let $C \in \AA$ be a pseudocircle with minimal weight $w(V_C)$.  Let
  $I_1, I_2, I_3$ and $J_1, J_2, J_3$ denote the $3$ color classes of a proper
  $3$-coloring of $G[V_I]$ and $G[V_O]$, respectively (Claim 1).  For
  $(i,j) \in \{1,2,3\}^2$, let $I_{i,j}:=I_i\cup J_j$ and let
  $X_{i,j} \subseteq V_C$ denote the set of vertices on $C$ with no neighbor
  in $I_{i,j}$. The subgraph $G[X_{i,j}]$ of $G[V_C]$ is $2$-colorable (Claim
  2). Let $X_{i,j}^1, X_{i,j}^2$ denote the color classes of such a coloring,
  and define independent sets $I_{i,j,k}:=I_{i,j} \cup X_{i,j}^k$ in $G$ for
  $k=1,2$.
	
  With $\textbf{I}$ we denote the random independent set $I_{i,j,k}$ with
  $(i,j,k)$ being chosen from the uniform distribution on
  $\{1,2,3\}\times\{1,2,3\}\times\{1,2\}$.  In the following we bound the
  expected weight $\mathbb{E}(w(\mathbf{I}))$.
	
  For every vertex $x \in V_C$, we have $x \in X_{i,j}$ if and only if none of
  the two neighbors $x^O$, $x^I$ of $x$ in $V_O$ respectively $V_I$ lie in
  $I_i$ respectively $J_j$. Since $i$, $j$ and $k$ are sampled independently,
  we conclude
  $$
  \mathbb{P}(x \in \mathbf{I}) = \frac{1}{2}\mathbb{P}(x \in X_{i,j})
  =\frac{1}{2}\mathbb{P}(x^O \notin I_i)\mathbb{P}(x^I \notin J_j)=
  \frac{1}{2}\Big(\frac{2}{3}\Big)^2=\frac{2}{9}.
  $$

  This implies that
  $$
  \mathbb{E}(w(\mathbf{I}))= \mathbb{E}(w(I_i \cup
  J_j))+\mathbb{E}(w(X_{i,j}^k)) =\frac{1}{3}(w(V_G)-w(V_C))+\frac{2}{9} \cdot
  w(V_C) =\frac{1}{3}w(V_G)-\frac{1}{9}w(V_C).
  $$
  Since $C$ was chosen as a pseudocircle of minimum weight, and since
  $\sum_{C' \in \AA}{w(V_{C'})}=2w(V_G)$, we conclude that
  $w(V_C) \le \frac{2}{n}w(V_G)$ and hence
  $\mathbb{E}(w(\mathbf{I})) \ge (\frac{1}{3}-\frac{2}{9n})w(V_G)$. Since
  $\mathbf{I}$ is ranging in the independent sets of $G$, this implies the
  existence of an independent set with total weight at least
  $(\frac{1}{3}-\frac{2}{9n})w(V_G)$.
\end{proof}

It is well known that the fractional chromatic number can be obtained as the
optimal value of the linear program
$$
\min \mathbf{1}\cdot x \hbox{\quad subject to\quad } Mx \geq \mathbf{1},
\quad x \geq 0
$$
where $M$ is the incidence matrix of vertices versus independent sets.  The
dual of the program is
$ \max \mathbf{1}\cdot w \hbox{ subject to } M^T w \leq \mathbf{1},\,\,w\geq
0$. Here $w$ can be interpreted as a weighting on the vertices. If $w$ is an
optimal weighting for this program, then $\chi_f(G) = w(V_G)$. With Claim~3 we
get	
$1 \geq \mathbb{E}(w(\mathbf{I})) \geq (\frac{1}{3}-\frac{2}{9n})w(V_G)$.
Hence, $\chi_f(G) \le \frac{1}{\frac{1}{3}-\frac{2}{9n}}=3+\frac{6}{3n-2}$.
\end{proof}

We note that for 4-vertex-critical graphs $G$, the following simple bound on
the fractional chromatic number further improves the bound given in
Theorem~\ref{theorem:chifrac_3_2n}.

\begin{proposition}
  
  If $G$ is a $4$-vertex-critical graph on $v$ vertices, then
  $\chi_f(G) \le 3+\frac{3}{v-1}$.
\end{proposition}
\begin{proof}
  We show that $G$ admits a $(v-1)$-fold coloring using $3v$ colors, which
  will imply $\chi_{v-1}(G) \le 3v$ and hence
  $\chi_f(G) \le \frac{\chi_{v-1}(G)}{v-1}\le \frac{3v}{v-1}=3+\frac{3}{v-1}$.
    
  The coloring can be obtained as follows: For every vertex $x \in V(G)$, fix
  a proper $3$-coloring
  $c_x:V(G) \setminus \{x\} \rightarrow \{C_{1,x},C_{2,x},C_{3,x}\}$ of the
  vertices in $G-x$ (which exists since $G$ is $4$-vertex-critical). Here,
  $\{C_{1,x},C_{2,x},C_{3,x}\}$ is a set of $3$ colors chosen such that these
  color-sets are pairwise disjoint for different vertices $x$.
    
  We now define a $(v-1)$-fold coloring of $G$ by assigning to every
  $w \in V(G)$ the following set of $v-1$ colors
  $\{c_x(w)|x \in V(G), x \neq w\}$. Since every $c_x$ is a proper coloring of
  $G$, these color-sets are disjoint for adjacent vertices in
  $G$. Furthermore, the coloring uses only colors in
  $\{\, C_{1,x},C_{2,x},C_{3,x} \mid x \in V(G) \,\}$, so $3v$ colors in total, and this
  proves the above claim and concludes the proof.
\end{proof}

\subsection{Arrangements with dense intersection graphs}
Given an arrangement $\mathcal{A}$ of pseudocircles, the
\emph{intersection graph} of $\mathcal{A}$ is the simple graph
$H_{\mathcal{A}}$ with the pseudocircles in $\mathcal{A}$ as the vertex-set in
which two distinct pseudocircles $C_1, C_2 \in \mathcal{A}$ share an edge if
and only if they cross. Using this notion, we see that intersecting
arrangements of pseudocircles are exactly the arrangements whose
intersection graph is a complete graph. Looking at
Theorem~\ref{theorem:chifrac_3_2n}, we were able to show that the fractional
chromatic number of such arrangements is close to $3$. In this section we
discuss possible generalizations of this result by extending this bound to
arrangements $\mathcal{A}$ for which $H_{\mathcal{A}}$ is sufficiently
dense. In particular we have the following question.
\begin{question}\label{question:largemindegree}
  For $k \in \mathbb{N}$, let $\chi_{\ge k}$ denote the supremum of
  $\chi_f(G)$ over all arrangement graphs~$G$ of arrangements $\mathcal{A}$ of
  pseudocircles such that the minimum degree $\delta(H_{\mathcal{A}})$ is at least
  $k$. Is it true that $\chi_{\ge k} \rightarrow 3$ for
  $k \rightarrow \infty$?
\end{question}
In the following, we show two weaker statements related to this question. The
first one shows that if we require the minimum degree in the intersection
graph of an arrangement to be sufficiently large compared to $n$, then we can
indeed conclude that the fractional chromatic number of the arrangement graph
is close to $3$. The second statement answers a relaxed version of
Question~\ref{question:largemindegree} by showing that for large minimum
degree in the intersection graph, the inverse independence ratio
$\frac{|V(G)|}{\alpha(G)}$ of the arrangement graph $G$ approaches $3$.
\begin{theorem}\label{theorem:largedegree}
  Let $d>\frac{1}{2}$ and $n \in \mathbb{N}$. Let $\mathcal{A}$ be a simple
  arrangement of $n$ pseudocircles such that $\delta(H_{\mathcal{A}}) \ge
  dn$. Then for the arrangement graph $G$ of $\mathcal{A}$, we have
  $\chi_f(G) \le \frac{3}{2d-1}$.
\end{theorem}
\begin{proof}
  The proof is similar to the one of Theorem~\ref{theorem:chifrac_3_2n}, and
  we borrow the notations from that proof. For a fixed pseudocircle $C \in \AA$, we
  further define $D_C \subseteq V_G\setminus V_C$ as the union of $V_{C'}$
  over all $C' \in \AA$ for which $C'$ and $C$ are disjoint. Also the
  following two claims hold for every choice of $C \in \AA$, with
  word-to-word the same proofs as for the according claims in the proof of
  Theorem~\ref{theorem:chifrac_3_2n}.
	
\paragraph{Claim 1.}
The graph $G-(V_C \cup D_C)$ is $3$-colorable.

\paragraph{Claim 2.}
The graph $G[V_C]$ is $2$-colorable.

\paragraph{Claim 3.} For every weighting $w:V_G \rightarrow [0,\infty)$ there
is an independent set $I$ of $G$ such that
$w(I) \ge (\frac{1}{3}-\frac{2(1-d)}{3})w(V_G)$.
\begin{proof}
  Fix $C \in \AA$ as a pseudocircle minimizing $w(V_C)+3w(D_C)$. In the
  following we fix some notation analogous to the one in the proof of
  Theorem~\ref{theorem:chifrac_3_2n}: We denote by $I_1, I_2, I_3$ and
  $J_1, J_2, J_3$ the color classes of a $3$-coloring of
  $G[V_I]-(V_I \cap D_C)$ and $G[V_O]-(V_O \cap D_C)$, respectively (which
  exist by Claim 1).  For $(i,j) \in \{1,2,3\}^2$, we denote again
  $I_{i,j}:=I_i\cup J_j$ and by $X_{i,j} \subseteq V_C$ the set of vertices on
  $C$ with no neighbor in $I_{i,j}$. Let $X_{i,j}^1, X_{i,j}^2$ denote the
  color classes of a $2$-coloring of $G[X_{i,j}] \subseteq G[V_C]$, and define
  independent sets $I_{i,j,k}:=I_{i,j} \cup X_{i,j}^k$ in $G$ for $k=1,2$.
		
  Again we let $\textbf{I}$ denote the random set $I_{i,j,k}$ where $(i,j,k)$
  is chosen uniformly at random from $\{1,2,3\}\times\{1,2,3\}\times\{1,2\}$.
		
  Every vertex $x \in V_C$ belongs to $X_{i,j}$ for at least 4 different
  choices of $(i,j)$. If $x$ has neighbors in $I_a$ and $J_b$, then it belongs
  to $X_{i,j}$ for $i \in \{1,2,3\}\setminus\{a\}$ and
  $j \in \{1,2,3\}\setminus\{b\}$.  Therefore,	
  $$
  \mathbb{E}(w(X_{i,j}^k))=
  \frac{1}{3}\cdot\frac{1}{3}\cdot\frac{1}{2}\sum_{i',j',k'} w(X_{i',j'}^{k'})
  = \frac{1}{18}\sum_{i',j'} w(X_{i',j'}) \geq \frac{1}{18}\cdot 4w(V_C) =
  \frac{2}{9}w(V_C).
  $$
  This implies that
  $$
  \mathbb{E}(w(\mathbf{I}))= \mathbb{E}(w(I_i \cup
  J_j))+\mathbb{E}(w(X_{i,j}^k)) \geq \frac{1}{3}(w(V_G)-w(V_C \cup
  D_C))+\frac{2}{9} \cdot w(V_C)
  $$\vskip-5mm
  $$
  =\frac{1}{3}w(V_G)-\frac{1}{9}(w(V_C)+3w(D_C))
  $$
  Since $C$ was chosen as a pseudocircle minimizing $w(V_C)+3w(D_C)$, we have
  $w(V_C)+3w(D_C) \le \frac{1}{n}\sum_{C' \in \AA}{(w(V_{C'})+3w(D_{C'}))}.$
  Let $v$ be a vertex in the intersection of two pseudocircles $C_1$ and $C_2$. For $i=1,2$
  pseudocircle $C_i$ is disjoint from at most $(n-1) -dn = (1-d)n-1$ other pseudocircles.
  Hence, $v$ is in at most $2(1-d)n -2$ sets $D_{C'}$ and we get
  $$
  \sum_{C' \in \AA}{(w(V_{C'})+3w(D_{C'}))} \leq
  \sum_{C' \in \AA}w(V_{C'})+3\sum_{C' \in \AA}w(D_{C'}) \leq
  $$\vskip-3mm
  $$  
  2\sum_{v \in V(G)}{w(v)} + (2(1-d)n -2)3\sum_{v \in V(G)}{w(v)} \leq
  6(1-d)n\cdot w(V_G).
  $$	
  Consequently $w(V_C)+3w(D_C) \le 6(1-d)w(V_G)$ and 
  $\mathbb{E}(w(\mathbf{I})) \ge (\frac{1}{3}-\frac{2(1-d)}{3})w(V_G)$. This
  implies the existence of an independent set with total weight at least
  $(\frac{1}{3}-\frac{2(1-d)}{3})w(V_G)$.
\end{proof}
	
Just as in the proof of Theorem~\ref{theorem:chifrac_3_2n} we express the
fractional chromatic number as the optimal value of the linear program
$ \max \mathbf{1}\cdot w \hbox{ subject to } M^T w \leq \mathbf{1},\,\,w\geq
0$ where $M$ is the incidence matrix of vertices versus independent sets. As
previously, Claim~3 now directly yields that
$\chi_f(G) \le 1/(\frac{1}{3}-\frac{2(1-d)}{3}) = \frac{3}{2d-1}$.
This concludes the proof of Theorem \ref{theorem:largedegree}.
\end{proof}

\begin{proposition}
  Let $G$ be the arrangement graph of a simple arrangement $\mathcal{A}$ of
  pseudocircles with $\delta(H_{\mathcal{A}}) \ge 2$. Then we
  have $\frac{|V(G)|}{\alpha(G)} \le 3+\frac{3}{\delta(H_{\mathcal{A}})-1}$.
\end{proposition}
\begin{proof}
  Let $C_0$ and $C_1$ be pseudocircles not belonging to $\mathcal{A}$, such
  that $C_0$ contains all pseudocircles of $\mathcal{A}$ and $C_1$ in its
  exterior, while $C_1$ has all pseudocircles in $\mathcal{A}$ and $C_0$ in
  its interior. By the Sweeping Theorem of Snoeyink and Hershberger
  \cite[Theorem~3.1]{SnoeyinkHershberger1991} there is a linear ordering
  $v_1,\ldots,v_{|V(G)|}$ of the vertices of $G$ such that each
  pseudocircle $C \in \mathcal{A}$ contains a unique vertex $v_C \in V_C$ with
  precisely $2$ predecessors on $C$ in this ordering, while all vertices in
  $V_C\setminus \{v_C\}$ are preceded by at most one other vertex on $C$.  It
  is now clear that the graph $G':=G-\{v_C|C \in \mathcal{A}\}$ is
  $2$-degenerate (since in the induced acyclic orientation of $G'$, every
  vertex has at most one in-edge on each of its two circles, and so the maximum
  in-degree in this orientation is at most $2$). Hence $G'$ is properly
  $3$-colorable by the greedy algorithm. Thus,
  $\alpha(G) \ge \alpha(G') \ge \frac{1}{3}(|V(G)|-|\mathcal{A}|)$. Since
  $\delta(H_{\mathcal{A}}) = k \ge 2$, every pseudocircle contains at least $2k$
  vertices, and hence we have $|V(G)| \ge k |\mathcal{A}|$. We finally
  conclude that
  $$
  \alpha(G) \ge
  \frac{1}{3}\left(|V(G)|-\frac{1}{k}|V(G)|\right)=\frac{k-1}{3k}|V(G)|.
  $$
\end{proof}

\section{Fractionally 3-color\-able 4-edge-critical planar graphs}
\label{sec:gimbel}

On the basis of the \emph{database of pseudocircles}~\cite{pseudocircles_website}
we could compute $\chi$ and $\chi_f$ exhaustively for small arrangements\footnote{
  Computing the fractional chromatic number of a graph is \NP-hard in general
  \cite{LundYannakakis1994}.  For our computations we formulated a linear
  program which we then solved using the MIP solver Gurobi.}.
We found the arrangement depicted in
Figure~\ref{fig:n5_pwi_chi4_1} with $\chi=4$ and $\chi_f=3$. This is a
counterexample to Conjecture~3.2 in Gimbel et al.~\cite{GimbelKLT2019}.

Extending the experiments to small 4-regular planar graphs we found that there
are precisely~17 \mbox{4-regular} planar graphs on 18 vertices with $\chi=4$
and $\chi_f=3$.  They are minimal in the sense that there are no 4-regular
graphs on $n \le 17$ vertices with $\chi=4$ and $\chi_f=3$.  Each of these~17
graphs is 4-vertex-critical and the one depicted in Figure~\ref{fig:crowning1}
is even 4-edge-critical.

Starting with a triangular face in the $4$-edge-critical $4$-regular graph of
Figure~\ref{fig:crowning1} and repeatedly applying Koester's crowning
operation as illustrated in Figure~\ref{fig:crowning2} (which by definition
preserves the existence of a facial triangle), we can deduce the following
theorem.

\begin{theorem} \label{thm:chif3}
  There exists an infinite family of $4$-edge-critical $4$-regular
  planar graphs~$G$ with fractional chromatic number $\chi_f(G)=3$.
\end{theorem}

We prepare the proof of the above result with some background on Koester's
crowing operation from~\cite{Koester1990}.  For a $4$-regular plane graph $G$
and a face $\pentagon$ of odd degree in $G$, we denote by~$\crown{G}{\pentagon}$ the plane graph
obtained by applying the crowning operation to $\pentagon$ in~$G$.
Figure~\ref{fig:krown} shows how to apply the crowning to a triangle and a
pentagon respectively, the general case should be deducible.
Koester proved the following:

% in einem figure environment mit caption
   \calc_figscale{20}
    \begin{figure}[htb]
    \centerline{\input{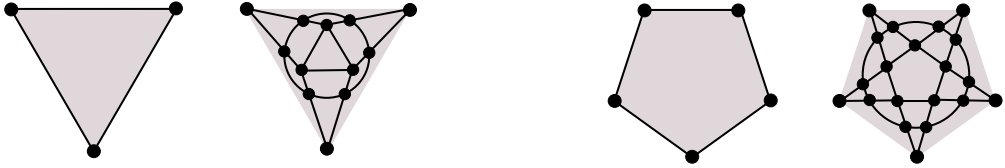}}
    \caption{Crowning of a triangle and a pentagon.\label{fig:krown}}
    \end{figure}

\begin{proposition}[\cite{Koester1990}]\label{prop:koester_edge_critical}
  Let $G$ be a $4$-regular plane graph with a facial triangle $T$. If $G$ is
  $4$-edge-critical, then so is $\crown{G}{T}$.
\end{proposition}

Via the following lemma, we can use Koester's crowning operation to extend the
example from Figure~\ref{fig:crowning1} to an infinite family of $4$-regular
$4$-edge-critical planar graphs with fractional chromatic number $3$.

\begin{figure}[htb]
  \hbox{} \hfill
  \begin{subfigure}[b]{.4\textwidth}
    \centering \includegraphics[page=1,width=\textwidth]{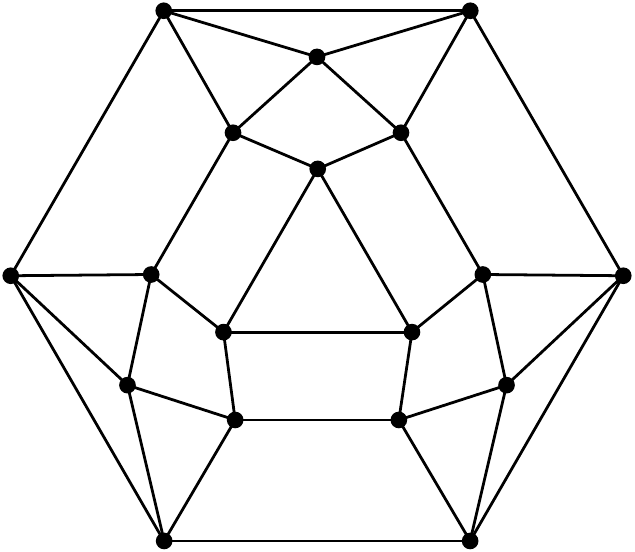}
    \caption{}
    \label{fig:crowning1}
  \end{subfigure}
  \hfill
  \begin{subfigure}[b]{.4\textwidth}
    \centering \includegraphics[page=2,width=\textwidth]{figs/crowning}
    \caption{}
    \label{fig:crowning2}
  \end{subfigure}
  \hfill \hbox{}
  \caption{ \subref{fig:crowning1}~A 4-edge-critical 4-regular 18-vertex
    planar graph with $\chi = 4$ and $\chi_f=3$ and \subref{fig:crowning2}~the
    crowning extension at its center triangular face.  }
  \label{fig:crowning}
\end{figure}

\begin{lemma}\label{lem:crowning_chi_f}
  Let $G$ be a $4$-regular plane graph with a facial triangle $T$.  If
  $\chi_f(G)=3$, then $\chi_f(\crown{G}{T})=3$.
\end{lemma}
\begin{proof}
  If $\chi_f(G)=3$, then it follows from the representation of $\chi_f(G)$ as
  the optimal value of a rational linear program that there exists
  $b \in \mathbb{N}$ such that $G$ has a $b$-coloring using $3b$ colors. For
  every vertex $v \in V(G)$, let $c(v) \in \binom{[3b]}{b}$ be the assigned
  sets of colors. Let $T=uvw$, then we know that $c(u), c(v), c(w)$ must be
  pairwise disjoint and hence form a partition of $\{1,\ldots,3b\}$. Let
  $c(u)=A_1$, $c(v)=A_2$, and $c(w)=A_3$. It is easy to see that the subgraph
  of $\crown{G}{T}$ induced by the vertices $u,v,w$ and the nine new vertices in
  $V(\crown{G}{T}) \setminus V(G)$ is 3-colorable such that the colors of $u,v,w$
  are pairwise distinct. By appropriately replacing the 3 colors by $A_1,A_2,A_3$
  we obtain a $b$-coloring of $\crown{G}{T}$ with $3b$ colors. This proves
  $\chi_f(\crown{G}{T}) \le 3$, now $\chi_f(\crown{G}{T}) = 3$ follows because $\crown{G}{T}$
  contains a triangle.
\end{proof}

Starting with a facial triangle in the $4$-regular $4$-edge-critical graph
of Figure~\ref{fig:crowning} and repeating the crowning operation
(which by definition preserves the existence of a facial triangle), by
Lemma~\ref{lem:crowning_chi_f} and
Proposition~\ref{prop:koester_edge_critical} we obtain an infinite family of
$4$-edge-critical $4$-regular planar graphs $G$ with fractional chromatic
number $\chi_f(G)=3$. This proves Theorem~\ref{thm:chif3}.

\section{Discussion}
\label{sec:discussion}

With Theorem~\ref{thm:3_colorable_arr} we gave a proof of
Conjecture~\ref{conj:FHNS} for $\triangle$-saturated great-pseudocircle
arrangements. While this is a very small subclass of great-pseudocircle
arrangements, it is reasonable to think of it as a ``hard'' class for
3-coloring.  The rationale for such thoughts is that triangles restrict the
freedom of extending partial colorings.  Our computational data indicates that
sufficiently large intersecting pseudocircle arrangements that are
\emph{diamond-free}, i.e., no two triangles of the arrangement share an edge,
are also 3-colorable.  Computations also suggest that sufficiently large
great-pseudocircle arrangements have \emph{antipodal colorings}, i.e.,
3-colorings where antipodal points have the same color.  Based on the
experimental data we propose the following strengthened variants of
Conjecture~\ref{conj:FHNS}.

\begin{conjecture}\label{conjecture:strenghened}
The following three statements hold:
  \begin{enumerate}[(a)]
  \item
    \label{conjecture:diamondfree_pwi_chi3}
    Every simple diamond-free intersecting arrangement of $ n \ge 6 $
    pseudocircles is 3-colorable.
  \item
    \label{conjecture:large_pwi_chi3}
    Every simple intersecting arrangement of sufficiently many pseudocircles
    is 3-colorable.
  \item
    \label{conjecture:antipodal3colorable}
    Every simple arrangement of $n \ge 7$ great-pseudocircles has an antipodal
    3-coloring.
  \end{enumerate}
\end{conjecture}

{\small
	\bibliographystyle{alphaabbrv-url}
	\bibliography{bibliography}
}
\end{document}